\documentclass[11pt,a4paper]{article}
\usepackage{geometry}
\usepackage{stackrel}
\usepackage{bbold}

\usepackage{amsmath,amsthm}
\usepackage{amssymb}
\usepackage{hyperref}

\usepackage{enumerate}

\usepackage{graphicx}

\newtheorem{definition}{Definition}

\newtheorem{theorem}{Theorem}[section]
\newtheorem{corollary}[theorem]{Corollary}
\newtheorem{lemma}[theorem]{Lemma}
\newtheorem{proposition}[theorem]{Proposition}

\theoremstyle{definition}

\newtheorem{remark}[theorem]{Remark}

\numberwithin{equation}{section}

\def\punpfrac#1#2{(#1)/(#2)}
\def\uncfrac#1#2{#1/#2}
\def\unfrac#1#2{#1/#2}
\def\prefrac#1#2{\tfrac{1}{#2}#1}
\def\prepfrac#1#2{\tfrac1{#2}(#1)}

\begin{document}

\baselineskip=17pt

\title{Linear fractional transformations and nonlinear leaping convergents of some continued fractions}

\author{Christopher Havens, Stefano Barbero, Umberto Cerruti, Nadir Murru \\
Department of Mathematics, University of Turin\\
Via Carlo Alberto 10, 10123, Torino, Italy}

\date{}

\maketitle

\begin{abstract}
For $\alpha_0 = \left[a_0, a_1, \ldots\right]$ an infinite continued fraction and $\sigma$ a
linear fractional transformation, we study the
continued fraction expansion of $\sigma(\alpha_0)$ and its convergents. 
We provide the continued fraction expansion of
$\sigma(\alpha_0)$ for four general families of continued fractions and when
$\left|\det \sigma\right| = 2$. We also find nonlinear recurrence
relations among the convergents of $\sigma(\alpha_0)$ which allow us to
highlight relations between convergents of $\alpha_0$ and $\sigma(\alpha_0)$.
Finally, we apply our results to some special and well-studied continued
fractions, like Hurwitzian and Tasoevian ones, giving a first study about
leaping convergents having steps provided by nonlinear functions.
\end{abstract}

\section{Introduction}

Continued fractions, which are classical objects in number theory, represent
any real number $\alpha_0$ by means of a sequence $a_i$ of integers:
\begin{equation*}\label{cf}\alpha_0 = a_0 + \cfrac{1}{a_1 + \cfrac{1}{a_2 + \cfrac{1}{\ddots}}}.\end{equation*} 
The integers $a_i$, called the \emph{partial quotients}, can be found by
\[
\begin{cases} a_k = [\alpha_k] \quad\hbox{(that is, the integer part of $\alpha_k$)}
\cr \alpha_{k+1} = \cfrac{1}{\alpha_k-a_k}
\quad \text{if $\alpha_k$ is not an integer} \end{cases} \quad\hbox{for } k = 0, 1, 2,
\ldots.
\]
The \emph{$n$-th convergent}  of $\alpha_{0}$ is the rational number
\begin{equation*}\label{conv}\frac{p_{n}}{q_{n}}= a_0 + \cfrac{1}{a_1 +
\cfrac{1}{a_2 + \cfrac{1}{\ddots+\cfrac{1}{
a_{n-1}+\cfrac{1}{a_{n}}}}}}.
\end{equation*}
For short, we will write
\begin{equation*}\label{cfcomp} \alpha_{0}=[a_0, a_1, a_2,
\ldots,a_{n},\ldots],\quad
\frac{p_{n}}{q_{n}}=[a_{0},a_{1},\ldots,a_{n}]\end{equation*} 
and we call the $n$-th \emph{tail} of $\alpha_0$ the continued fraction
\begin{equation*}\label{tail} t_{n}=[a_{n+1},a_{n+2},\ldots].\end{equation*}

Let us consider a linear fractional transformation
$\sigma(x)=\punpfrac{Ax+B}{Cx+D}$ with $A, B, C, D \in \mathbb{Z}$. The
determinant of $\sigma$ is $\det \sigma := AD - BC$. It is well known that
$\alpha_0$ and $\sigma(\alpha_0)$ have the same tails, for $n$ sufficiently
large, if and only if $\left|\det \sigma\right| = 1$ (see for example \cite{Rock}). In this
paper, we show a relation between the tail of $\sigma(\alpha_0)$
and the tail of $\alpha_0$ when $\left|\det \sigma\right|=2$ and the partial quotients of
$\alpha_0$ have some specific behaviors. When the value of $\left|\det \sigma\right|$
increases, some numerical experiments show that the previous relations are very
hard to find. 
Hence, in the general case, that is, when $\left|\det \sigma\right|=d$ with $d\neq0,\pm1$, it
is very hard to predict the behaviour of $\sigma(\alpha_{0})$. 
The proofs in this paper are greatly based on the fact that $\sigma$ has determinant $\pm 2$, in fact the main idea is to write the matrix $\begin{pmatrix} A & B \cr C & D \end{pmatrix}$ as an unimodular matrix multiplied by auxiliary matrices that commutate with the $R-L$ matrices used by Raney \cite{Ran} for representing continued fractions, like in proposition \ref{Sdec}.
When $\left| \det \sigma \right| > 2$, similar relations are more hard to give and we think that other ways should be used for approaching the problem for general values of determinants. However, the results of this paper could help in finding a solution to the general case.

The Gosper algorithm is a useful tool for evaluating a linear fractional
transformation of a given continued fraction and it can be also used for
determining sums and products of continued fractions. The original algorithm
was presented in \cite[Item 101B]{Gosper2}. Some descriptions about the Gosper
algorithm can be also found in  \cite[pp. 347--352]{Fowler}, \cite{Hall},
\cite{Liardet} and \cite{Vardi}.
Another standard algorithm, based on the representation of a continued fraction
as a product of powers of particular matrices, is due to Raney \cite{Ran}.
These algorithms allow us to obtain the partial quotients of $\sigma(\alpha_0)$
starting from the partial quotients of $\alpha_0$; however they are not
suitable for determining general relations, since they would need to use the values
of the partial quotients of $\alpha_0$.
Some authors have studied the linear fractional transformation of continued
fractions with bounded partial quotients \cite{Lag}, \cite{Pan}. In \cite{Lee},
the author presented an algorithm for determining the continued fraction
expansion of a linear fractional transformation of power series. See also
\cite{Astels}, \cite{Divis}.

The results obtained in this paper can also be used to determine the linear
fractional transformation of certain Hurwitzian and Tasoevian continued fractions,
as well as studying their leaping convergents.

\begin{definition}
	A \emph{Hurwitz continued fraction} is a quasi-periodic continued fraction of the kind
	$$[a_0,\ldots, a_r, \overline{f_1(k),\ldots, f_s(k)}]_{k=h}^{+\infty} = [a_0,\ldots, a_r, f_1(h),\ldots, f_s(h), f_1(h+1), \ldots, f_s(h+1),\ldots],$$
	for certain integers $r, s, h$, where $f_1(k), \ldots f_s(k)$ are polynomials in $k$ with rational coefficients.
\end{definition}

\begin{definition}
	A \emph{Tasoev continued fraction} is a quasi-periodic continued fraction of the kind
	$$[a_0,\ldots, a_r, \overline{g_1(k),\ldots, g_s(k)}]_{k=h}^{+\infty} = [a_0,\ldots, a_r, g_1(h),\ldots, g_s(h), g_1(h+1), \ldots, g_s(h+1),\ldots],$$
	for certain integers $r, s, h$, where $g_1(k), \ldots g_s(k)$ are exponentials in $k$ with rational coefficients.
\end{definition}

A famous example of Hurwitz continued fraction is this expansion of $e$:
$$e = [2, 1, 2, 1, 1, 4, 1, \ldots] = [2, \overline{1, 2k, 1}]_{k=1}^{+\infty}.$$
There are other many known Hurwitz continued fractions that represent
\emph{e-type} numbers like
$$e^{1/m} = [1, \overline{m(2k-1)-1, 1, 1}]_{k=1}^{+\infty}
\quad\hbox{and}\quad
\cfrac{e-1}{e+1} = [0, \overline{4 k + 2}]_{k=0}^{+\infty}.$$
Some numbers expressed by means of the tangent function also have a
quasi-periodic continued fraction expansion, for example

$$\tan 1 = [1, \overline{2 k - 1, 1}]_{k=1}^{+\infty}$$
and
$$\sqrt{\frac{m}{n}}\tan\frac{1}{\sqrt{mn}} = [0, n-1, \overline{1, (4k-1)m-2, 1, (4k+1)n-2}]_{k=1}^{+\infty}.$$

The properties of Hurwitz continued fractions is a
classical research field; see for example \cite{Davis}, \cite{Lehmer1},
\cite{Lehmer}, \cite{Matt}, \cite{Walters}.
 Many recent studies have dealt with their leaping convergents. Usually,
given the convergents $p_i/q_i$ of a continued fraction, the leaping
convergents are $p_{ri+j}/q_{ri+j}$, with $r\geq2$ and $0\leq j \leq r-1$ fixed
integers, for $i=0, 1, 2,\ldots{}$. In the case of Hurwitz continued fractions,
the leaping step $r$ is usually (but not necessarily) chosen to be equal to the
length of the period.

In \cite{Elsner}, the author found a recurrence formula for the leaping
convergents $p_{3i+1}/q_{3i+1}$ of the continued fraction of $e$. A similar
result was found by Komatsu \cite{Komatsu1} for the leaping convergents
$p_{3i}/q_{3i}$ of $e^{1/m}$. In \cite{Komatsu2}, Komatsu extended these
results to more general Hurwitz continued fractions, and in \cite{Elsner2} a
recurrence formula was also found for non-regular Hurwitz continued fractions.
In \cite{Komatsu7}, Komatsu proved other recurrence formulas and closed
forms for leaping convergents of generalized tanh-type Hurwitz continued
fractions. In several works, the goal was to find closed forms for the
convergents of some Hurwitz and Tasoev continued fractions and study their
leaping convergents; see \cite{Komatsu6}, \cite{Komatsu4},
\cite{Komatsu5}, \cite{Komatsu3}, \cite{Komatsu8}, and \cite{Komatsu9}.
As a consequence of our results, we find relations involving leaping
convergents of some Hurwitz and Tasoev continued fractions having steps
provided by nonlinear functions.

\section{Linear fractional transformations of some continued fractions}
In this section we study the relations between the continued fraction
expansions of $\alpha_0$ and $\sigma(\alpha_0)$, when $\left| \det \sigma
\right|=2$.
We deal with infinite continued fractions of four different types.
\begin{definition} \label{cftypes} 
\begin{itemize}
	\item[1)] The continued fraction
	\begin{equation*}\label{cf1}
	\left[e_{0},e_{1},e_{2},\ldots,e_{i},\ldots{}\right],
	\end{equation*}
	where $e_{i}$, $i=0,1,\ldots{}$, are even positive integers, is called type 1 (CF1).
	\item[2)] The continued fraction
	\begin{equation*}\label{cf2}
	\left[d_{0},d_{1},d_{2},\ldots,d_{i},\ldots{}\right],
	\end{equation*}
	where $d_{i}$, $i=0,1,\ldots{}$, are odd positive integers, is called type 2 (CF2).
	\item[3)] The continued fraction
	\begin{equation*}\label{cf3}
	\left[d_{0},e_{1},d_{2},e_{3}\ldots,d_{2i},e_{2i+1},\ldots\right],
	\end{equation*}
	where $d_{2i}$, $e_{2i+1}$, $i=0,1,\ldots{}$, are odd and even positive
integers, respectively, is called type 3 (CF3).
	\item[4)] The continued fraction
	\begin{equation*}\label{cf4}
	\left[e_{0},d_{1},e_{2},d_{3}\ldots,e_{2i},d_{2i+1}\ldots\right],
	\end{equation*}
	where $e_{2i}$, $d_{2i+1}$, $i=0,1,\ldots{}$, are even and odd positive
integers, respectively, is called type 4 (CF4).
\end{itemize}
\end{definition}
In the following, we will take $d_{0}$ and $e_{0}$ to be positive integers
without loss of generality.
Now we give a classification of linear fractional transformations $\sigma$
with integer coefficients and $\left|\det \sigma\right|=2$.
\begin{proposition}\label{Sdec}
	Let $\sigma$ be a linear fractional transformation, i.e.,
$\sigma(x)=\frac{Ax+B}{Cx+D}$, where $x$ is a real number and $A,B,C,D\in
\mathbb{Z}$. Let $S$ be the associated matrix $$S =\begin{pmatrix}  A & B \cr C
& D \end{pmatrix}.$$ If $\left|\det \sigma\right| = \left|\det S\right| = 2$, then
\begin{itemize}	
\item[1)]
$S=TM$, 
\item[2)] $S=TMR$ or
\item[3)] $S=TMRJ$,
\end{itemize}
where $$M=\begin{pmatrix}  1 & 1 \cr 1 & -1 \end{pmatrix},\quad R=\begin{pmatrix}
1 & 1 \cr 0 & 1 \end{pmatrix},\quad J=\begin{pmatrix}  0 & 1 \cr 1 &
0\end{pmatrix}$$ and $T$ is a certain unimodular matrix with integer entries.
\end{proposition}
\begin{proof}
We observe that only one of the following situations can hold:
\begin{itemize}
	\item  The entries of the first row of $S$ have the same parity and the
entries of the second row have the same parity. In this case 
$$S = TM\quad\hbox{with}\quad T=\begin{pmatrix}  \frac{A+B}{2} & \frac{A-B}{2} \cr \frac{C+D}{2} & \frac{C-D}{2} \end{pmatrix}.$$
	\item In the first column of $S$ there is at least one odd entry and
the remaining entries are all even. In this case $$S = TMR\quad\hbox{with}\quad
T=\begin{pmatrix}  \frac{B}{2} & \frac{2A-B}{2} \cr \frac{D}{2} &
\frac{2C-D}{2} \end{pmatrix}.$$
	\item In the second column of $S$ there is at least one odd entry and
the remaining entries are all even. In this case $$S = TMRJ\quad\hbox{with}\quad
T=\begin{pmatrix}  \frac{A}{2} & \frac{2B-A}{2} \cr \frac{C}{2} &
\frac{2D-C}{2} \end{pmatrix}.$$
	\end{itemize}
Other possibilities for $S$ are forbidden by the condition
$\left|\det(S)\right|=2$.
\end{proof}

\begin{lemma}\label{matrix rel} Let us consider the four matrices
	$$M=\begin{pmatrix}  1 & 1 \cr 1 & -1 \end{pmatrix},\quad R=\begin{pmatrix}  1 & 1 \cr 0 & 1 \end{pmatrix},\quad L=\begin{pmatrix}  1 & 0 \cr 1 & 1 \end{pmatrix},\quad J=\begin{pmatrix}  0 & 1 \cr 1 & 0\end{pmatrix}.$$
Then the following equalities hold:
\begin{itemize}
\item [a)] $$JR^{h}=L^{h}J,\quad JL^{h}=R^{h}J,$$
$$\begin{pmatrix} 0 & 2 \cr 1 & 0 \end{pmatrix}L^h =R^{2h}\begin{pmatrix} 0 & 2
\cr 1 & 0 \end{pmatrix},\quad \begin{pmatrix} 0 & 1 \cr 2 & 0 \end{pmatrix}R^h
= L^{2h}\begin{pmatrix} 0 & 1 \cr 2 & 0 \end{pmatrix}$$
for any integer $h$;
\item[b)] $$MR^h = RL^{\frac{h-1}{2}}\begin{pmatrix} 0 & 2 \cr 1 & 0
\end{pmatrix},\quad \begin{pmatrix} 0 & 2 \cr 1 & 0 \end{pmatrix} R^h =
L^{\frac{h-1}{2}}\begin{pmatrix} 0 & 2 \cr 1 & 1 \end{pmatrix},$$
$$\begin{pmatrix} 0 & 1 \cr 2 & 0 \end{pmatrix}L^h =R^{\frac{h-1}{2}}LM, \quad
\begin{pmatrix} 0 & 2 \cr 1 & 1 \end{pmatrix}L^h =
RLR^{\frac{h-1}{2}}\begin{pmatrix} 0 & 1 \cr 2 & 0 \end{pmatrix}$$
for any odd integer $h$;
\item[c)] $$MR^{h}=RL^{\frac{h-2}{2}}\begin{pmatrix} 0 & 2 \cr 1 & 1 \end{pmatrix},\quad \begin{pmatrix} 0 & 2 \cr 1 & 0 \end{pmatrix} R^h = L^{\frac{h}{2}}\begin{pmatrix} 0 & 2 \cr 1 & 0 \end{pmatrix},$$
  $$\begin{pmatrix} 0 & 2 \cr 1 & 1 \end{pmatrix}L^h = RLR^{\frac{h-2}{2}}LM $$
  for any even integer $h$.
\end{itemize}
\end{lemma}
\begin{proof}
	All these identities are straightforward to check.
\end{proof}
	Now, we are ready to describe the continued fraction expansion of
$\sigma(\alpha_0)$, when $\alpha_0$ is of type CF1, CF2, CF3 or CF4. Let us
recall that a continued fraction $[a_0, a_1, a_2, a_3,\ldots{}]$ is equivalent to
the matrix product
	$$R^{a_0}L^{a_1}R^{a_2}L^{a_3}\ldots;$$ 
see \cite{Ran}.	We will use this fact, Proposition \ref{Sdec}, and Lemma
\ref{matrix rel} for proving the following theorem.
	\begin{theorem}\label{transform}
	\begin{itemize}
			\item [1)] If $\alpha_0$ is a continued fraction of
type CF1, then it is equivalent to one of the following matrix products:
				 \begin{equation}\label{1.1}T\prod_{k=1}^{+\infty}\left(RL^{\frac{e_{k-1}-2}{2}}RLR^{\frac{e_{k}-2}{2}}L\right),\end{equation}
				 \begin{equation}\label{1.2} TR\prod_{k=1}^{+\infty}\left(L^{\frac{e_{2k-2}}{2}}R^{2e_{2k-1}}\right),\end{equation}
				 \begin{equation}\label{1.3}TR^{2e_{0}+1}\prod_{k=1}^{+\infty}\left(L^{\frac{e_{2k-1}}{2}}R^{e_{2k}}\right).\end{equation}
			\item [2)] If $\alpha_0$ is a continued fraction of
type CF2, then it is equivalent to one of the following matrix products:
 \begin{equation}\label{2.1}T\prod_{k=1}^{+\infty}\left(RL^{\frac{d_{3k-3}-1}{2}}R^{2d_{3k-2}}L^{\frac{d_{3k-1}-1}{2}}RLR^{\frac{d_{3k}-1}{2}}L^{2d_{3k+1}}R^{\frac{d_{3k+2}-1}{2}}L\right),\end{equation}
			 \begin{equation}\label{2.2}TRL^{\frac{d_{0}-1}{2}}R\prod_{k=1}^{+\infty}\left(LR^{\frac{d_{3k-2}-1}{2}}L^{2d_{3k-1}}R^{\frac{d_{3k}-1}{2}}LRL^{\frac{d_{3k+1}-1}{2}}R^{2d_{3k+2}}L^{\frac{d_{3k+3}-1}{2}}R\right),\end{equation}
		\small
		\begin{equation}\label{2.3} TR^{2d_{0}+1}L^{\frac{d_{1}-1}{2}}R\prod_{k=1}^{+\infty}\left(LR^{\frac{d_{3k-1}-1}{2}}L^{2d_{3k}}R^{\frac{d_{3k+1}-1}{2}}LRL^{\frac{d_{3k+2}-1}{2}}R^{2d_{3k+3}}L^{\frac{d_{3k+4}-1}{2}}R\right).\end{equation}
	\normalsize
		\item [3)] If $\alpha_0$ is a continued fraction of type CF3,
then it is equivalent to one of the following matrix products:
		 \begin{equation*}\label{3.1}T\prod_{k=1}^{+\infty}\left(RL^{\frac{d_{4k-4}-1}{2}}R^{2e_{4k-3}}L^{\frac{d_{4k-2}-1}{2}}RLR^{\frac{e_{4k-1}-2}{2}}L\right),\end{equation*}
		 \begin{equation*}\label{3.2}TRL^{\frac{d_{0}-1}{2}}RLR^{\frac{e_{1}-2}{2}}L\prod_{k=1}^{+\infty}\left(RL^{\frac{d_{4k-2}-1}{2}}R^{2e_{4k-1}}L^{\frac{d_{4k}-1}{2}}RLR^{\frac{e_{4k+1}-2}{2}}L\right),\end{equation*}
		 
		 \begin{equation*}\label{3.3}TR^{2d_{0}+1}\prod_{k=1}^{+\infty}\left(L^{\frac{e_{2k-1}}{2}}R^{2d_{2k}}\right).\end{equation*}
	\item [4)] If $\alpha_0$ is a continued fraction of type CF4, then it is equivalent to one of the following matrix products:
	 \begin{equation*}\label{4.1}TRL^{\frac{e_{0}-2}{2}}R\prod_{k=1}^{+\infty}\left(LR^{\frac{d_{4k-3}-1}{2}}L^{2e_{4k-2}}R^{\frac{d_{4k-1}-1}{2}}LRL^{\frac{e_{4k}-2}{2}}R\right),\end{equation*}
	 \begin{equation}\label{4.2}TR\prod_{k=1}^{+\infty}\left(L^{\frac{e_{2k-2}}{2}}R^{2d_{2k-1}}\right),\end{equation}
	 \begin{equation*}\label{4.3}TR^{2e_{0}+1}L^{\frac{d_{1}-1}{2}}RLR^{\frac{e_{2}-2}{2}}L\prod_{k=1}^{+\infty}\left(RL^{\frac{d_{4k-1}-1}{2}}R^{2e_{4k}}L^{\frac{d_{4k+1}-1}{2}}RL
	R^{\frac{e_{4k+2}-2}{2}}L \right).\end{equation*}
\end{itemize}
\end{theorem}
\begin{proof}
We prove the theorem for continued fractions of type CF1 and we give a sketch
of the proof for continued fractions of type CF2. The results concerning
continued fractions of type CF3 and CF4 can be easily proved with the same
techniques. To prove (\ref{1.1}), we consider the infinite product
$$MR^{e_{0}}L^{e_{1}}R^{e_{2}}L^{e_{3}}\ldots$$
and we use the first relation in Lemma \ref{matrix rel} part c), obtaining
$$RL^{\frac{e_{0}-2}{2}}\begin{pmatrix} 0 & 2 \cr 1 & 1 \end{pmatrix}L^{e_{1}}R^{e_{2}}L^{e_{3}}\ldots.$$
Using the third relation in Lemma \ref{matrix rel} part c), we have
$$RL^{\frac{e_{0}-2}{2}}RLR^{\frac{e_{1}-2}{2}}LMR^{e_{2}}L^{e_{3}}\ldots.$$
Now, we may iterate the applications of these two matrix relations, so that the
matrix $M$ vanishes on the right, and we find the product on the right of $T$ in
(\ref{1.1}).
Now, we deal with (\ref{1.2}) and consider the infinite product
$$MR^{e_{0}+1}L^{e_{1}}R^{e_{2}}L^{e_{3}}\ldots.$$
Applying the first relation in Lemma \ref{matrix rel} part b), since $e_{0}+1$ is an odd integer, we get
$$RL^{\frac{e_{0}}{2}}\begin{pmatrix} 0 & 2 \cr 1 & 0 \end{pmatrix}L^{e_{1}}R^{e_{2}}L^{e_{3}}\ldots.$$
Thanks to the third relation in Lemma \ref{matrix rel} part a), we have
$$RL^{\frac{e_{0}}{2}}R^{2e_{1}}\begin{pmatrix} 0 & 2 \cr 1 & 0 \end{pmatrix}R^{e_{2}}L^{e_{3}}\ldots,$$
then, using the second relation in Lemma \ref{matrix rel} part c), we find
$$RL^{\frac{e_{0}}{2}}R^{2e_{1}}L^{\frac{e_{2}}{2}}\begin{pmatrix} 0 & 2 \cr 1 & 0 \end{pmatrix}L^{e_{3}}\ldots.$$
Finally, we can apply the last two relations so that the matrix
$\left(\begin{smallmatrix} 0 & 2 \cr 1 & 0 \end{smallmatrix}\right)$ vanishes on the right,
obtaining the product on the right of $T$ in (\ref{1.2}).
 For a proof of (\ref{1.3}), we consider the infinite product
$$MRJR^{e_{0}}L^{e_{1}}R^{e_{2}}L^{e_{3}}\ldots.$$
Then we start using the first two relations in Lemma \ref{matrix rel} part a),
so that the matrix $J$ vanishes on the right and we find 
$$MRL^{e_{0}}R^{e_{1}}L^{e_{2}}R^{e_{3}}L^{e_{4}}\ldots.$$
Consequently, we apply the first relation in Lemma \ref{matrix rel} part b),
which gives
$$R\begin{pmatrix} 0 & 2 \cr 1 & 0 \end{pmatrix}L^{e_{0}}R^{e_{1}}L^{e_{2}}R^{e_{3}}L^{e_{4}}\ldots .$$
So we can infinitely iterate the third relation in Lemma \ref{matrix rel} part
a) and
the second relation in Lemma \ref{matrix rel} part c), so that the matrix
$\left(\begin{smallmatrix} 0 & 2 \cr 1 & 0 \end{smallmatrix}\right)$ vanishes on the right and we
retrieve the product on the right of $T$ in (\ref{1.3}).
In the case of contniued fractions of type CF2, for proving (\ref{2.1}), we
take into account that
$$MR^{d_{0}}L^{d_{1}}\cdots=RL^{\frac{d_{0}-1}{2}}\begin{pmatrix} 0 & 2 \cr 1 & 0 \end{pmatrix}R^{d_{2}}L^{d_{3}}\cdots=RL^{\frac{d_{0}-1}{2}}R^{2d_{1}}\begin{pmatrix} 0 & 2 \cr 1 & 0 \end{pmatrix}R^{d_{2}}L^{d_{3}}\cdots=$$
$$=RL^{\frac{d_{0}-1}{2}}R^{2d_{1}}L^{\frac{d_{2}-1}{2}}\begin{pmatrix} 0 & 2 \cr 1 & 1 \end{pmatrix}L^{d_{3}}R^{d_{4}}\cdots=RL^{\frac{d_{0}-1}{2}}R^{2d_{1}}L^{\frac{d_{2}-1}{2}}RLR^{\frac{d_{3}-1}{2}}\begin{pmatrix} 0 & 1 \cr 2 & 0 \end{pmatrix}R^{d_{4}}L^{d_{5}}\cdots=$$
$$=RL^{\frac{d_{0}-1}{2}}R^{2d_{1}}L^{\frac{d_{2}-1}{2}}RLR^{\frac{d_{3}-1}{2}}L^{2d_{4}}\begin{pmatrix} 0 & 1 \cr 2 & 0 \end{pmatrix}L^{d_{5}}R^{d_{6}}\cdots=$$
$$=RL^{\frac{d_{0}-1}{2}}R^{2d_{1}}L^{\frac{d_{2}-1}{2}}RLR^{\frac{d_{3}-1}{2}}L^{2d_{4}}R^{\frac{d_{5}-1}{2}}LMR^{d_{6}}L^{d_{7}}\cdots .$$
Similarly, for the purpose of proving (\ref{2.2}), we observe that
$$MR^{d_{0}+1}L^{d_{1}}\cdots=
RL^{\frac{d_{0}-1}{2}}\begin{pmatrix} 0 & 2 \cr 1 & 1 \end{pmatrix}L^{d_{1}}R^{d_{2}}\cdots=RL^{\frac{d_{0}-1}{2}}RLR^{\frac{d_{1}-1}{2}}\begin{pmatrix} 0 & 1 \cr 2 & 0 \end{pmatrix}R^{d_{2}}L^{d_{3}}\cdots=$$
$$=RL^{\frac{d_{0}-1}{2}}RLR^{\frac{d_{1}-1}{2}}L^{2d_{2}}\begin{pmatrix} 0 & 1 \cr 2 & 0 \end{pmatrix}L^{d_{3}}R^{d_{4}}\cdots=RL^{\frac{d_{0}-1}{2}}RLR^{\frac{d_{1}-1}{2}}L^{2d_{2}}R^{\frac{d_{3}-1}{2}}LMR^{d_{4}}L^{d_{5}}\cdots .$$
Finally, for a proof of (\ref{2.3}), we have
$$MRJR^{d_{0}}L^{d_{1}}R^{d_{2}}L^{d_{3}}R^{d_{4}}L^{d_{5}}\cdots=MRL^{d_{0}}R^{d_{1}}L^{d_{2}}R^{d_{3}}L^{d_{4}}R^{d_{5}}\cdots=$$
$$=R\begin{pmatrix} 0 & 2 \cr 1 & 0 \end{pmatrix}L^{d_{0}}R^{d_{1}}\cdots=R^{2d_{0}+1}\begin{pmatrix} 0 & 2 \cr 1 & 0 \end{pmatrix}R^{d_{1}}\cdots=R^{2d_{0}+1}L^{\frac{d_{1}-1}{2}}\begin{pmatrix} 0 & 2 \cr 1 & 1 \end{pmatrix}L^{d_{2}}R^{d_{3}}\cdots=$$
$$=R^{2d_{0}+1}L^{\frac{d_{1}-1}{2}}RLR^{\frac{d_{2}-1}{2}}\begin{pmatrix} 0 & 1 \cr 2 & 0 \end{pmatrix}R^{d_{3}}\cdots=R^{2d_{0}+1}L^{\frac{d_{1}-1}{2}}RLR^{\frac{d_{2}-1}{2}}L^{2d_{3}}\begin{pmatrix} 0 & 1 \cr 2 & 0 \end{pmatrix}L^{d_{4}}\cdots=$$
$$=R^{2d_{0}+1}L^{\frac{d_{1}-1}{2}}RLR^{\frac{d_{2}-1}{2}}L^{2d_{3}}R^{\frac{d_{4}-1}{2}}LMR^{d_{5}}L^{d_{6}}\cdots{} .
$$
\end{proof}
Since $T$ is an unimodular matrix, for $n$ sufficiently large, we find an
explicit expression for the tails of $\sigma(\alpha_0)$, where $\alpha_0$ is a
continued fraction of type CF1, CF2, CF3 or CF4. We will describe these tails
in the next corollary.
\begin{corollary} \label{corcf}
\begin{itemize}
\item If $\alpha_0$ is a continued fraction of type CF1 and $e_{k-1}\geq4$ for all $k\geq k_{0}\geq 1$, then the tail of $\sigma(\alpha_0)$ is 
\begin{equation}\label{t1.1}
\left[\overline{\frac{e_{k-1}-2}{2},1,1}\right]_{k=k_{0}}^{+\infty};
\end{equation}
otherwise it is
\begin{equation}\label{t1.2}
\left[\overline{\frac{e_{2k-2}}{2},2e_{2k-1}}\right]_{k=k_{0}}^{+\infty},
\end{equation}
\begin{equation}\label{t1.3}
\left[\overline{\frac{e_{2k-1}}{2},2e_{2k}}\right]_{k=k_{0}}^{+\infty}.
\end{equation}

\item If $\alpha_0$ is a continued fraction of type CF2 and $d_{3k-3}\geq3$ and $d_{3k-1}\geq3$ for all $k\geq k_{0}\geq 1$, then the tail of $\sigma(\alpha_0)$ is
\begin{equation}\label{t2.1}
\left[\overline{\frac{d_{3k-3}-1}{2},2d_{3k-2},\frac{d_{3k-1}-1}{2},1,1}\right]_{k=k_{0}}^{+\infty};
\end{equation}
if $d_{3k-2}\geq3$ and $d_{3k}\geq3$ for all $k\geq k_{0}\geq 1$, then it is
\begin{equation}\label{t2.2}
\left[\overline{\frac{d_{3k-2}-1}{2},2d_{3k-1},\frac{d_{3k}-1}{2},1,1}\right]_{k= k_{0}}^{+\infty};
\end{equation}
if $d_{3k-1}\geq3$ and $d_{3k+1}\geq3$ for all $k\geq k_{0}\geq 1$, then it is
\begin{equation}\label{t2.3}
\left[\overline{\frac{d_{3k-1}-1}{2},2d_{3k},\frac{d_{3k+1}-1}{2},1,1}\right]_{k=k_{0}}^{+\infty}.
\end{equation}

\item If $\alpha_0$ is a continued fraction of type CF3 and $d_{4k-4}\geq3$,
$d_{4k-2}\geq3$ and $e_{4k-1}\geq4$ for all $k\geq k_{0}\geq 1$, then the tail
of $\sigma(\alpha_0)$ is
\begin{equation}\label{t3.1}
\left[\overline{\frac{d_{4k-4}-1}{2},2e_{4k-3},\frac{d_{4k-2}-1}{2},1,1,\frac{e_{4k-1}-2}{2},1,1}\right]_{k=k_{0}}^{+\infty};
\end{equation}
if $d_{4k-2}\geq3$, $d_{4k}\geq3$ and $e_{4k+1}\geq4$ for all $k\geq k_{0}\geq1$, then it is
\begin{equation}\label{t3.2}
\left[\overline{\frac{d_{4k-2}-1}{2},2e_{4k-1},\frac{d_{4k}-1}{2},1,1,\frac{e_{4k+1}-2}{2},1,1}\right]_{k=k_{0}}^{+\infty};
\end{equation}
otherwise it is
\begin{equation}\label{t3.3}
\left[\overline{\frac{e_{2k-1}}{2},2d_{2k}}\right]_{k=k_{0}}^{+\infty}.
\end{equation}

\item  If $\alpha_0$ is a continued fraction of type CF4 and $d_{4k-3}\geq3$,
$d_{4k-1}\geq3$ and $e_{4k}\geq4$ for all $k\geq k_{0}\geq1$, then the tail of
$\sigma(\alpha_0)$ is
\begin{equation}\label{t4.1}
\left[\overline{\frac{d_{4k-3}-1}{2},2e_{4k-2},\frac{d_{4k-1}-1}{2},1,1,\frac{e_{4k}-2}{2},1,1}\right]_{k=k_{0}}^{+\infty};
\end{equation}
if $d_{4k-1}\geq3$, $d_{4k+1}\geq3$ and $e_{4k+2}\geq4$ for all $k\geq k_{0}\geq1$, then
\begin{equation*}\label{t4.2}
\left[\overline{\frac{e_{2k-2}}{2},2d_{2k-1}}\right]_{k=k_{0}}^{+\infty},
\end{equation*}
\begin{equation}\label{t4.3}
\left[\overline{\frac{d_{4k-1}-1}{2},2e_{4k},\frac{d_{4k+1}-1}{2},1,1,\frac{e_{4k+2}-2}{2},1,1}\right]_{k=k_{0}}^{+\infty}.
\end{equation}
\end{itemize}
\end{corollary}
\begin{proof}
	The proof is straightforward, thanks to the matrix representations
described in Theorem \ref{transform}.
\end{proof}
\begin{remark}
When $T$ has nonnegative entries and $\det T=1$, then $T$
can be also written as  a product of  powers of $R$ and $L$ matrices, with
exponents given by nonnegative integers; see \cite{Ran}. Thus, in this case,
by Theorem \ref{transform}, the relations for the tail of $\sigma(\alpha_0)$
hold for $n=1$.  \end{remark}

\section{Relations for the convergents}
Let $x$ be a continued fraction and $H_n = \uncfrac{u_n}{v_n}$ its convergents.
Given $B_n = \sigma(H_n)$ and $\beta = \sigma(x)$, we clearly have
$$\lim_{n\rightarrow +\infty} B_n = \beta.$$ 
In general, the $B_n$'s do not coincide with the convergents
$\uncfrac{U_n}{V_n}$ of $\beta$. However, when $\left|\det \sigma\right| = 2$ and $x$ is a
continued fraction of type CF1, CF2, CF3 or CF4, we have
$B_{n}=\unfrac{U_{t}}{V_{t}}$ for certain values of $n$ and $t$ defined by
nonlinear functions. 
We start proving a theorem which highlights nonlinear recurrence relations
involving the sequences $\{U_{t}\}_{t\geq0}$ and $\{V_{t}\}_{t\geq0}$.
\begin{theorem}\label{recrel}
Using the notation of Corollary \ref{corcf}, let us consider $p_{0} = k_0 - 1$
and define the functions $$g(p) =p_{0}+p +
2\left\lfloor \frac{p}{3}\right \rfloor, \quad h(p) =p_{0}+2p - 1 + \sin\left(
\prefrac{(p+1)\pi}{2} \right),$$
	$$v(p) =2^{\sin\left(\left(p+1+\left\lfloor \frac{p+2}{3}\right\rfloor\right)\frac{\pi}{2}\right)},\quad 
	 z(p) =2^{\sin\left(\prefrac{(p+2)\pi}{2}\right)},$$
	for $p\geq2$. Then the sequences $\{U_{t}\}_{t\geq0}$ and $\{V_{t}\}_{t\geq0}$ satisfy one of the following
nonlinear recurrence relations:
\begin{itemize}
\item[1)]	
	\begin{equation}\label{rec1}
	U_{s(p)}=e_{l_{1}(p)}U_{s(p-1)}+U_{s(p-2)},\quad V_{s(p)}=e_{l_{1}(p)}V_{s(p-1)}+V_{s(p-2)},
	\end{equation}
	when $x$ is a continued fraction of type CF1, with tail given by \ref{t1.1}, where $s(p)=p_{0}+3p$ and $l_{1}(p)=k_{0}+p-2$; 
\item [2)]
\begin{equation}\label{rec2}
	U_{g(p)}=G(p)U_{g(p-1)}+v(p)U_{g(p-2)},\quad V_{g(p)}=G(p)V_{g(p-1)}+v(p)V_{g(p-2)},
\end{equation}
when $x$ is a continued fraction of type CF2 expansion, where
\begin{equation*}\label{Gp}
G(p)=d_{l_{2}(p)}2^{\sin\left(\left(p+2+\left\lfloor \frac{p+1}{3}\right\rfloor\right)\frac{\pi}{2}\right)}
\end{equation*}
and 
\begin{equation*} 
l_{2}(p)=\begin{cases}3k_{0}+p-4 \quad  \text{for tail (\ref{t2.1})}\\
3k_{0}+p-3 \quad\text{for tail (\ref{t2.2})}\\
	3k_{0}+p-2 \quad \text{for tail (\ref{t2.3})};
\end{cases}
\end{equation*}
\item [3)]
\begin{equation}\label{rec3}
U_{h(p)}=H(p)U_{h(p-1)}+z(p)U_{h(p-2)},\quad V_{h(p)}=H(p)V_{h(p-1)}+z(p)V_{h(p-2)},
\end{equation}
when $x$ is a continued fraction of type CF3 or CF4, where
\begin{equation*} \label{Hp}
H(p) =a_{l_{3,4}(p)}2^{\sin\left(\left(p+2+\left\lfloor \frac{p+2}{4}\right\rfloor-\left\lfloor \frac{p}{4}\right\rfloor\right)\frac{\pi}{2}\right)}
\end{equation*}
and 
\begin{equation*} 
l_{3,4}(p)=\begin{cases}4k_{0}+p-5 \quad  \text{for tail (\ref{t3.1})}\\
4k_{0}+p-3 \quad\text{for tail (\ref{t3.2})}\\
4k_{0}+p-4 \quad \text{for tail (\ref{t4.1})}\\
4k_{0}+p-2 \quad \text{for tail (\ref{t4.3})},
\end{cases}
\end{equation*}
 with $a_{l_{3,4}(p)}=e_{l_{3,4}(p)}$ if $p$ is even and $a_{l_{3,4}(p)}=d_{l_{3,4}(p)}$ if $p$ is odd.
\end{itemize}
\end{theorem}
\begin{proof}
\begin{enumerate}
\item The following equalities clearly hold, for all $p \geq 2$:
\begin{equation}\label{rec1.1}
U_{p_{0}+3p}=U_{p_{0}+3p-1}+U_{p_{0}+3p-2},\quad U_{p_{0}+3p-1}=U_{p_{0}+3p-2}+U_{p_{0}+3p-3},
\end{equation}
\begin{equation}\label{rec1.2}
U_{p_{0}+3p-2}=\prepfrac{e_{k_{0}+p-2}-2}{2}U_{p_{0}+3p-3}+U_{p_{0}+3p-4},
\end{equation}
\begin{equation}\label{rec1.3}
 U_{p_{0}+3p-3}=U_{p_{0}+3p-4}+U_{p_{0}+3p-5},\quad 
 U_{p_{0}+3p-4}=U_{p_{0}+3p-5}+U_{p_{0}+3p-6}.
\end{equation}
From (\ref{rec1.1}) and (\ref{rec1.3}), we find
\begin{equation}\label{rec1.5}
2U_{p_{0}+3p-4}=U_{p_{0}+3p-3}+U_{p_{0}+3p-6},
\end{equation}
\begin{equation}\label{rec1.4}
U_{p_{0}+3p}=U_{p_{0}+3p-3}+2U_{p_{0}+3p-2}.
\end{equation}
Finally, using (\ref{rec1.2}) and (\ref{rec1.5}) in (\ref{rec1.4}), we obtain
the first recurrence of (\ref{rec1}); the second one can be proved similarly.
\item We deal with three different possibilities: $p=3m$, $p=3m+1$ or $p=3m+2$, where $m$ is a positive integer.
First we observe that
\begin{equation*}\label{gpval1}
g(3m)=p_{0}+5m,\quad g(3m+1)=p_{0}+5m+1,\quad g(3m+2)=p_{0}+5m+2,
\end{equation*}
\begin{equation*}\label{gpval2}
g(3m-1)=p_{0}+5m-3, \quad g(3m-2)=p_{0}+5m-4,
\end{equation*}
\begin{equation*}\label{vpval}
v(3m)=2,\quad v(3m+1)=\tfrac{1}{2},\quad v(3m+2)=1,
\end{equation*}
\begin{equation*}\label{Gpval}
G(3m)=d_{l_{2}(3m)},\quad G(3m+1)=\prefrac{d_{l_{2}(3m+1)}}{2},\quad G(3m+2)=2d_{l_{2}(3m+2)}.
\end{equation*}
Since $\sigma(x)$ has one of the tails (\ref{t2.1}), (\ref{t2.2}) or
(\ref{t2.3}) when $x$ is a continued fraction of type CF2, we find the
following relations: 
\begin{equation*}\label{3m1}
U_{p_{0}+5m-1}=U_{p_{0}+5m-2}+U_{p_{0}+5m-3},\quad U_{p_{0}+5m-2}=\prepfrac{d_{l_{2}(3m)}-1}{2}U_{p_{0}+5m-3}+U_{p_{0}+5m-4},
\end{equation*}
\begin{equation}\label{3m2}
U_{p_{0}+5m}=U_{p_{0}+5m-1}+U_{p_{0}+5m-2},\quad U_{p_{0}+5m+1}=\prepfrac{d_{l_{2}(3m+1)}-1}{2}U_{p_{0}+5m}+U_{p_{0}+5m-1},
\end{equation}
from which we have
\begin{multline}\label{3m}
U_{g(3m)}=U_{p_{0}+5m}=d_{l_{2}(3m)}U_{p_{0}+5m-3}+2U_{p_{0}+5m-4}
\\
=G(3m)
U_{g(3m-1)}+v(3m)U_{g(3m-2)},\end{multline}
\begin{equation}\label{3m3}
U_{p_{0}+5m-1}=\tfrac{1}{2}(U_{p_{0}+5m}+U_{p_{0}+5m-3}).
\end{equation}
Using (\ref{3m3}) and (\ref{3m2}), we get
\begin{multline*}
U_{g(3m+1)}=U_{p_{0}+5m+1}=\prefrac{\left(d_{l_{2}(3m+1)}U_{p_{0}+5m}+U_{p_{0}+5m-3}\right)}{2}\\=G(3m+1)U_{g(3m)}+v(3m+1)U_{g(3m-1)}.
\end{multline*}
Moreover, from
\begin{multline*}
$$U_{g(3m+2)}=U_{p_{0}+5m+2}=2d_{l_{2}(3m+2)}U_{p_{0}+5m}+U_{p_{0}+5m+1}
\\
=G(3m+2)U_{g(3m+1)}+v(3m+2)U_{g(3m)},$$
\end{multline*}
the proof of the first recurrence in (\ref{rec2}) is complete. The second recurrence can be proved similarly.
\item The proof of recurrences (\ref{rec3}) is similar, considering the four different possibilities $p=4m$, $p=4m+1$, $p=4m+2$ and $p=4m+3$, where $m$ is a positive integer.
\end{enumerate}
\end{proof}
Now we are ready to prove when the $B_{n}$'s are equal to some convergents of $\beta$.
\begin{theorem}\label{leaping}
Let $\sigma$ be a linear fractional transformation with $\left|\det
\sigma\right| = 2$ and let $x$ be a continued fraction.  Given
$\{\unfrac{u_{t}}{v_{t}}\}_{t\geq0}$ the convergents of $x$,
$\{\unfrac{U_{t}}{V_{t}}\}_{t\geq0}$ the convergents of $\sigma(x)$, and
$B_{t}=\sigma(\unfrac{u_{t}}{v_{t}})$ for all $t\geq0$, we have that
\begin{itemize}
\item if $x$ is a continued fraction of type CF1 (with tail (\ref{t1.1})),
CF2, CF3 (with tail (\ref{t3.1}) or (\ref{t3.2})), or CF4 (with tail
(\ref{t4.1}) or (\ref{t4.3})), then one of the following equalities hold
\begin{equation}\label{eqconv1}
\frac{U_{s(p)}}{V_{s(p)}}=B_{l_{1}(p)}, 
\end{equation}
\begin{equation}\label{eqconv2}
 \frac{U_{g(p)}}{V_{g(p)}}=B_{l_{2}(p)},
 \end{equation}
 \begin{equation}\label{eqconv3}
 \frac{U_{h(p)}}{V_{h(p)}}=B_{l_{3,4}(p)}
\end{equation}
for any $p\geq3$, where $g$, $s$, $h$, $l_1$, $l_2$, $l_{3,4}$ are the functions defined in Theorem \ref{recrel};
\item if $x$ is a continued fraction expansion of type CF1 (with tails
(\ref{t1.2}) or (\ref{t1.3})), CF3 (with tail (\ref{t3.3})), or CF4 (with tail (\ref{4.2})), we have 
\begin{equation}\label{eqconv4}
\frac{U_{p}}{V_{p}}=B_{p}
\end{equation}
for any $p$ sufficiently large.
\end{itemize}
\end{theorem} 
\begin{proof}
	We will only prove equality (\ref{eqconv2}), since the proofs for
(\ref{eqconv1}), (\ref{eqconv3}), and (\ref{eqconv4}) can be obtained 
similarly. Let $N_{l_{2}(p)}$ and $D_{l_{2}(p)}$ be the numerator and the
denominator of $B_{l_{2}(p)}$, respectively. Since
$B_{l_{2}(p)}=\sigma(H_{l_{2}(p)})$, we have 
	\begin{equation*}
	N_{l_{2}(p)}=d_{l_{2}(p)}N_{l_{2}(p-1)}+N_{l_{2}(p-2)},\quad D_{l_{2}(p)}=d_{l_{2}(p)}D_{l_{2}(p-1)}+D_{l_{2}(p-2)},
	\end{equation*}
 Now, we prove by induction that, if $p=3m$ or $p=3m+2$, then
\begin{equation}\label{ind1}
U_{g(p)}=N_{l_{2}(p)},\quad V_{g(p)}=D_{l_{2}(p)}; 
\end{equation}
if $p=3m+1$, then
\begin{equation}\label{ind2}
U_{g(p)}=\frac{N_{l_{2}(p)}}{2},\quad V_{g(p)}=\frac{D_{l_{2}(p)}}{2},
\end{equation}
with $m$ positive integer.
The inductive bases for (\ref{ind1}) and (\ref{ind2}) are the equalities
\begin{equation*}\label{indb1}
U_{g(3)}=U_{p_{0}+5}=N_{l_{2}(3)},\quad V_{g(3)}=V_{p_{0}+5}=D_{l_{2}(3)},
\end{equation*}
\begin{equation*}\label{indb2}
U_{g(5)}=U_{p_{0}+7}=N_{l_{2}(5)},\quad V_{g(5)}=V_{p_{0}+7}=D_{l_{2}(5)},
\end{equation*}
\begin{equation*}\label{indb3}
U_{g(4)}=U_{p_{0}+6}=\prefrac{N_{l_{2}(4)}}{2},\quad V_{g(4)}=V_{p_{0}+6}=\prefrac{D_{l_{2}(4)}}{2}.
\end{equation*}
These relations can be proved thanks to the matrix equalities
\begin{equation*}\label{meq1}
\begin{pmatrix} N_{l_{2}(4)} & N_{l_{2}(3)} \cr D_{l_{2}(4)} & D_{l_{2}(3)}\end{pmatrix}=SR^{d_{1}}\ldots L^{d_{l_{2}(4)}}=\begin{pmatrix} U_{p_{0}+5} & U_{p_{0}+6} \cr V_{p_{0}+5} & V_{p_{0}+6} \end{pmatrix}\begin{pmatrix} 0 & 1 \cr 2 & 0 \end{pmatrix} \quad \text{$p_{0}$ even},
\end{equation*}
\begin{equation*}\label{meq1.1}
\begin{pmatrix} N_{l_{2}(4)} & N_{l_{2}(5)} \cr D_{l_{2}(4)} & D_{l_{2}(5)}\end{pmatrix}=SR^{d_{0}}\ldots R^{d_{l_{2}(5)}}=\begin{pmatrix} U_{p_{0}+7} & U_{p_{0}+6} \cr V_{p_{0}+7} & V_{p_{0}+6} \end{pmatrix}\begin{pmatrix} 0 & 1 \cr 2 & 0 \end{pmatrix} \quad \text{$p_{0}$ even},
\end{equation*}
\begin{equation*}\label{meq2}
\begin{pmatrix} N_{l_{2}(3)} & N_{l_{2}(4)} \cr D_{l_{2}(3)} & D_{l_{2}(4)}\end{pmatrix}=SR^{d_{0}}\ldots R^{d_{l_{2}(4)}}=\begin{pmatrix} U_{p_{0}+6} & U_{p_{0}+5} \cr V_{p_{0}+6} & V_{p_{0}+5} \end{pmatrix}\begin{pmatrix} 0 & 2 \cr 1 & 0 \end{pmatrix} \quad \text{$p_{0}$ odd},
\end{equation*}
\begin{equation*}\label{meq2.2}
\begin{pmatrix} N_{l_{2}(5)} & N_{l_{2}(4)} \cr D_{l_{2}(5)} & D_{l_{2}(4)}\end{pmatrix}=SR^{d_{0}}\ldots L^{d_{l_{2}(5)}}=\begin{pmatrix} U_{p_{0}+6} & U_{p_{0}+7} \cr V_{p_{0}+6} & V_{p_{0}+7} \end{pmatrix}\begin{pmatrix} 0 & 2 \cr 1 & 0 \end{pmatrix} \quad \text{$p_{0}$ odd}.
\end{equation*}

In fact, by definition of $p_{0}$, we know that $p_{0}$ and $l_{2}(4)$
must have different parity, moreover, for $p_{0}$ even, we have
\begin{equation*}
SR^{d_{0}}\ldots L^{d_{l_{2}(4)}}=R^{a_{0}}\ldots R^{a_{p_{0}}}L^{\frac{d_{l_{2}(1)}-1}{2}}R^{2l_{2}(2)}L^{\frac{d_{l_{2}(3)}-1}{2}}RLR^{\frac{d_{l_{2}(4)}-1}{2}}\begin{pmatrix} 0 & 1 \cr 2 & 0 \end{pmatrix},
\end{equation*}
\begin{equation*}
SR^{d_{0}}\ldots R^{d_{l_{2}(5)}}=R^{a_{0}}\ldots R^{a_{p_{0}}}L^{\frac{d_{l_{2}(1)}-1}{2}}R^{2l_{2}(2)}L^{\frac{d_{l_{2}(3)}-1}{2}}RLR^{\frac{d_{l_{2}(4)}-1}{2}}L^{2d_{l_{2}(5)}}\begin{pmatrix} 0 & 1 \cr 2 & 0 \end{pmatrix},
\end{equation*}
while, for $p_{0}$ odd, 
\begin{equation*}
SR^{d_{0}}\ldots R^{d_{l_{2}(4)}}=R^{a_{0}}\ldots L^{a_{p_{0}}}R^{\frac{d_{l_{2}(1)}-1}{2}}L^{2l_{2}(2)}R^{\frac{d_{l_{2}(3)}-1}{2}}LRL^{\frac{d_{l_{2}(4)}-1}{2}}\begin{pmatrix} 0 & 2 \cr 1 & 0 \end{pmatrix},
\end{equation*}
\begin{equation*}
SR^{d_{0}}\ldots L^{d_{l_{2}(5)}}=R^{a_{0}}\ldots L^{a_{p_{0}}}R^{\frac{d_{l_{2}(1)}-1}{2}}L^{2l_{2}(2)}R^{\frac{d_{l_{2}(3)}-1}{2}}LRL^{\frac{d_{l_{2}(4)}-1}{2}}R^{2d_{l_{2}(5)}}\begin{pmatrix} 0 & 2 \cr 1 & 0 \end{pmatrix}.
\end{equation*}
Thus, if we suppose that (\ref{ind1}) and (\ref{ind2}) hold for all the
positive  integers $3,\ldots,p-1$, using the recurrence relations (\ref{rec2}),
we find 
$$U_{g(p)}=d_{l_{2}(p)}U_{g(p-1)}+2U_{g(p-2)}=d_{l_{2}}N_{l_{2}(p-1)}+N_{l_{2}(p-2)}=N_{l_{2}(p)}\quad
\text{if $p=3m$},$$
$$U_{g(p)}=\prefrac{d_{l_{2}(p)}}{2}U_{g(p-1)}+\tfrac{1}{2}U_{g(p-2)}
=\tfrac{1}{2}d_{l_{2}}N_{l_{2}(p-1)}+\tfrac{1}{2}N_{l_{2}(p-2)}=\prefrac{N_{l_{2}(p)}}{2}\quad
\text{if $p=3m+1$},$$
$$U_{g(p)}=2d_{l_{2}(p)}U_{g(p-1)}+U_{g(p-2)}=d_{l_{2}}N_{l_{2}(p-1)}+N_{l_{2}(p-2)}=N_{l_{2}(p)}\quad
\text{if $p=3m+2$},$$
and the equivalent relations for $V_{g(p)}$ and $D_{l_{2}(p)}$ are also true.
\end{proof}

\section{Nonlinear leaping convergents of some Hurwitzian and Tasoev continued fractions}

In this section we apply the previous results to some quasi-periodic continued
fractions.
Many well-studied Hurwitzian and Tasoevian continued fractions are continued fractions of type CF1, CF2, CF3 or CF4. 

Let us consider the Hurwitz continued fraction 
\begin{equation}\label{h1}
h(a,n)=[\overline{a(1 + k n)}]_{k=0}^{+\infty},
\end{equation}
for $a, n$ positive integers, whose convergents will be denoted by $H_p(a,n)$,
for $p = 0, 1,\ldots .$ The identity  $$H_p(a,n)=\frac{\sum
_{i=0}^{\left\lfloor \frac{p+1}{2}\right\rfloor } a^{p-2 i+1} \binom{p-i+1}{i}
\prod _{k=i}^{-i+p} (k n+1)}{\sum _{i=0}^{\left\lfloor \frac{p}{2}\right\rfloor
} a^{p-2 i} \binom{p-i}{i} \prod _{k=i+1}^{p-i} (k n+1)}$$ can be retrieved as
a special case of Corollary 7 in \cite{McLaughlin}. 
Moreover, it is simple to show that 
$$B_{p}=\sigma(H_{p}(a,n))=\frac{\sum _{i=0}^{\left\lfloor
\frac{p}{2}\right\rfloor } a^{p-2 i} \left(Aa (i n+1)
\binom{p-i+1}{i}+B\binom{p-i}{i}\right) \prod _{k=i+1}^{p-i} (k
n+1)+A\delta(p)}{\sum _{i=0}^{\left\lfloor \frac{p}{2}\right\rfloor } a^{p-2 i}
\left(Ca (i n+1) \binom{p-i+1}{i}+D\binom{p-i}{i}\right) \prod _{k=i+1}^{p-i}
(k n+1)+C\delta(p)},$$ 
for $p=0,1,\ldots$, where $\delta(p)=\frac{1-(-1)^p}{2}.$ We now observe that
the continued fraction (\ref{h1}) is
\begin{itemize}
	\item of type CF1, when $a$ is even, and the possible tails of $\sigma(h(a,n))$ are given by
	\begin{equation}\label{ht1.1}
	\left[\overline{\frac{1}{2}(a(1+n(k-1))-2), 1, 1}\right]_{k=k_{0}}^{+\infty},
	\end{equation}
	\begin{equation*}\label{ht1.2}
	\left[\overline{\frac{a}{2}(1+n(2k-2)), 2a(1+(2k-1)n)}\right]_{k=k_{0}}^{+\infty},
	\end{equation*}
	\begin{equation*}\label{ht1.3}
	\left[\overline{\frac{a}{2}(1+n(2k-1)), 2a(1+2kn)}\right]_{k=k_{0}}^{+\infty};
	\end{equation*}
	
	\item of type CF2, when $a$ is odd and $n$ even, and the possible tails of $\sigma(h(a,n))$ are given by
	\begin{equation}\label{ht2.1}
	\left[\overline{\frac{(a(1+3n(k-1))-1)}{2}, 2a(1+n(3k-2)), \frac{(a(1+n(3k-1))-1)}{2}, 1, 1}\right]_{k=k_{0}}^{+\infty}, 
	\end{equation}
	\begin{equation}\label{ht2.2}
	\left[\overline{\frac{(a(1+n(3k-2))-1)}{2}, 2a(1+n(3k-1)), \frac{(a(1+3kn)-1)}{2}, 1, 1}\right]_{k=k_{0}}^{+\infty}, 
	\end{equation}
	\begin{equation}\label{ht2.3}
	\left[\overline{\frac{(a(1+n(3k-1))-1)}{2}, 2a(1+3kn), \frac{(a(1+n(3k+1))-1)}{2}, 1, 1}\right]_{k=k_{0}}^{+\infty};
	\end{equation}
	\item of type CF3, when $a$, $n$ are odd and greater than 3, and the possible tails of $\sigma(h(a,n))$ are given by
		\begin{equation*}\label{ht3.3}
	\left[\overline{\frac{a}{2}(1+n(2k-1)), 2a(1+2kn)}\right]_{k=k_{0}}^{+\infty}.
	\end{equation*}
	\begin{equation*}\label{ht3}
	\left[\overline{w_1, w_2, w_3, w_4, w_5, w_6, w_7, w_8}\right]_{k=k_{0}}^{+\infty},
	\end{equation*}
	where $w_{4}=w_{5}=w_{7}=w_{8}=1$ and 
	\begin{equation}\label{ht3.1}
	\begin{split}
	w_{1}=\frac{a(1+(4k-4)n)-1}{2},\quad&  w_{2}=2a(1+(4k-3)n) \\
	w_{3}=\frac{a(1+(4k-2)n)-1}{2},\quad&  w_{6}=\frac{a(1+(4k-1)n)-1}{2},
	\end{split}
	\end{equation}
	or 
	\begin{equation}\label{ht3.2}
	\begin{split}
	w_{1}=\frac{a(1+(4k-2)n)-1}{2},\quad & w_{2}=2a(1+(4k-1)n)\\
	w_{3}=\frac{a(1+4kn)-1}{2}, \quad & w_{6}=\frac{a(1+(4k+1)n)-1}{2}.
	\end{split}
	\end{equation}

\end{itemize}
Thus, if $\{\unfrac{U_{t}}{V_{t}}\}_{t\geq0}$ is the sequence of convergents of
$\sigma(h(a,n))$, from Theorem \ref{leaping} we have 
\begin{equation}\label{eqconvh1}
\frac{U_{s(p)}}{V_{s(p)}}=B_{l_{1}(p)} 
\end{equation}
for $a$ even and tail (\ref{ht1.1}), with $s(p)=p_{0}+3p$, $l_{1}(p)=k_{0}+p-2,$
\begin{equation*}\label{eqconvh2}
\frac{U_{g(p)}}{V_{g(p)}}=B_{l_{2}(p)}
\end{equation*}
for $a$ odd and $n$ even with $$g(p) =p_{0}+p + 2\left\lfloor \frac{p}{3}\right \rfloor,\quad
l_{2}(p)=\begin{cases}3k_{0}+p-4 \quad  \text{for tail (\ref{ht2.1})}\\
3k_{0}+p-3 \quad\text{for tail (\ref{ht2.2})}\\
3k_{0}+p-2 \quad \text{for tail (\ref{ht2.3})}
\end{cases},$$
 and, finally,
\begin{equation*}\label{eqconvh3}
\frac{U_{h(p)}}{V_{h(p)}}=B_{l_{3}(p)}
\end{equation*}
for $a$, $n$ odd and greater than 3, with 
$$l_{3}(p)=\begin{cases}4k_{0}+p-5 \quad  \text{for tail (\ref{ht3.1})}\\
4k_{0}+p-3 \quad\text{for tail (\ref{ht3.2}),}\\
\end{cases}\quad
h(p) =p_{0}+2p - 1 + \sin\left( \prefrac{(p+1)\pi}{2} \right).$$
When $a$ is even, $\sigma(h(a,n))$ has tail (\ref{ht1.1}) and if $a_{p_{0}}=a_{0}=1$, $k_{0}=1$,
 we have a Hurwitzian continued fraction of the form 
$$[1, \overline{\alpha k + \beta, 1, 1}]_{k=1}^{+\infty} \quad\hbox{with } \alpha = \frac{an}{2},\quad\beta = \frac{a(1 - n) - 2}{2}.$$
 Therefore, identity (\ref{eqconvh1}) becomes $$\frac{U_{3p}}{V_{3p}}=B_{p-1}.$$ 
 These Hurwitz continued fractions have been studied by Komatsu
\cite{Komatsu}, who found the combinatorial expression for the leaping
convergents $\unfrac{U_{3p}}{V_{3p}}$; see Corollary 1, \cite{Komatsu}.
Regarding Tasoevian continued fractions, we may consider two examples:
 \begin{equation}\label{tasoev1}
 t_1(u,a) = [\overline{u a^k}]_{k=1}^{+\infty},
 \end{equation}
 \begin{equation}\label{tasoev2}
 t_2(u,v,a) = [\overline{u a^k, v a^k}]_{k=1}^{+\infty},
 \end{equation}
 where $a$, $u$ and $v$ are positive integers. The continued fraction (\ref{tasoev1}) is 
 \begin{itemize}
 	\item of type CF1, when $u$ is even or $a$ is even, with $\sigma(t_{1}(u,a)) $ having tails, for $k_{0}\geq1$, given by 
$ \left[\overline{\prepfrac{ua^{k}-2}{2},1,1}\right]_{k=k_{0}}^{+\infty}$
 if we suppose that $ua\geq4$, or
 $\left[\overline{\prefrac{ua^{2k-1}}{2},2ua^{2k}}\right]_{k=k_{0}}^{+\infty}$
and $ \left[\overline{\prefrac{ua^{2k}}{2},2ua^{2k+1}}\right]_{k=k_{0}}^{+\infty};$
 \item of type CF2, when $u$ and $a$ are odd, with $\sigma(t_{1}(u,a)) $ having tails, for $k_{0}\geq 1$ and $ua\geq3$, given by
$$ \left[\overline{\prepfrac{ua^{3k-2}-1}{2},2ua^{3k-1},\prepfrac{ua^{3k}-1}{2},1,1}\right]_{k=k_{0}}^{+\infty},$$
$$ \left[\overline{\prepfrac{ua^{3k-1}-1}{2},2ua^{3k},\prepfrac{ua^{3k+1}-1}{2},1,1}\right]_{k= k_{0}}^{+\infty},$$
$$ \left[\overline{\prepfrac{ua^{3k}-1}{2},2ua^{3k+1},\prepfrac{ua^{3k+2}-1}{2},1,1}\right]_{k=k_{0}}^{+\infty}.$$
\end{itemize} 
Thus, for $\sigma(t_1(u,a))$ one of the identities (\ref{eqconv1}), (\ref{eqconv2}), (\ref{eqconv4})  holds.
 On the other hand, the continued fraction (\ref{tasoev2}) is
 \begin{itemize}
 \item of type CF1, when $a$ is even or $u$ and $v$ are even, with
$\sigma(t_{2}(u,v,a))$ having tails, for $k_{0}\geq1$, given by
$$ \left[\overline{\prepfrac{ua^{k}-2}{2},1,1,\prepfrac{va^{k}-2}{2},1,1}\right]_{k=k_{0}}^{+\infty}$$
 if we suppose that $ua\geq4$ and $va\geq4$, or
 $\left[\overline{\prefrac{ua^{k}}{2},2va^{k}}\right]_{k=k_{0}}^{+\infty}$ and
$ \left[\overline{\prefrac{va^{k}}{2},2ua^{k+1}}\right]_{k=k_{0}}^{+\infty};$
 \item  of type CF2, when $a$, $u$, $v$ are odd, with $\sigma(t_{2}(u,v,a))$ having tails, for $k_{0}\geq1$ and $ua\geq3$, $va\geq3$, given by
$$ \left[\overline{\prepfrac{ua^{3k-2}-1}{2},2va^{3k-2},\prepfrac{ua^{3k-1}-1}{2},1,1,\prepfrac{ua^{3k-1}-1}{2},2ua^{3k},\prepfrac{va^{3k}-1}{2},1,1}\right]_{k=k_{0}}^{+\infty},$$
$$ \left[\overline{\prepfrac{va^{3k-2}-1}{2},2ua^{3k-1},\prepfrac{va^{3k-1}-1}{2},1,1,\prepfrac{ua^{3k}-1}{2},2va^{3k},\prepfrac{va^{3k+1}-1}{2},1,1}\right]_{k= k_{0}}^{+\infty},$$
 $$\left[\overline{\prepfrac{ua^{3k-1}-1}{2},2va^{3k-1},\prepfrac{ua^{3k}-1}{2},1,1,\prepfrac{va^{3k}-1}{2},2ua^{3k+1},\prepfrac{va^{3k+1}-1}{2},1,1}\right]_{k=k_{0}}^{+\infty};$$
\item of type CF3, if $a$ and $u$ are odd and $v$ is even, with $\sigma(t_{2}(u,v,a))$ having tails, for $k_{0}\geq1$ and $ua\geq3$, $va\geq4$, given by
$$\left[\overline{\prepfrac{ua^{2k-1}-1}{2},2va^{2k-1},\prepfrac{ua^{2k}-1}{2},1,1,\prepfrac{va^{2k}-2}{2},1,1}\right]_{k=k_{0}}^{+\infty},$$
$$\left[\overline{\prepfrac{ua^{2k}-1}{2},2va^{2k},\prepfrac{ua^{2k+1}-1}{2},1,1,\prepfrac{va^{2k+1}-2}{2},1,1}\right]_{k=k_{0}}^{+\infty},$$
$$\left[\overline{\prefrac{va^{k}}{2},2ua^{k+1}}\right]_{k=k_{0}}^{+\infty};$$
\item of type CF4, if $a$, $v$ are odd and $u$ is even, with
$\sigma(t_{2}(u,v,a))$ having tails for $k_{0}\geq1$ and $ua\geq4$, $va\geq3$
given by
$$\left[\overline{\prepfrac{va^{2k-1}-1}{2},2ua^{2k},\prepfrac{va^{2k}-1}{2},1,1,\prepfrac{ua^{2k+1}-2}{2},1,1}\right]_{k=k_{0}}^{+\infty},$$
$$\left[\overline{\prefrac{ua^{k}}{2},2va^{k}}\right]_{k=k_{0}}^{+\infty},$$
$$\left[\overline{\prepfrac{va^{2k}-1}{2},2ua^{2k+1},\prepfrac{va^{2k+1}-1}{2},1,1,\prepfrac{ua^{2k+2}-2}{2},1,1}\right]_{k=k_{0}}^{+\infty}.$$
\end{itemize}
Therefore for $\sigma(t_2(u,v,a))$ one of the identities (\ref{eqconv1}), (\ref{eqconv2}), (\ref{eqconv3}), (\ref{eqconv4}) holds.

\end{document}